\documentclass[a4paper,11pt]{amsart}

\usepackage{fullpage}

\makeatletter
\@addtoreset{equation}{section}

\makeatother
\makeatletter

\usepackage[latin1]{inputenc}
\usepackage[english]{babel}
\usepackage{amsmath,amssymb,amsfonts,amsthm,mathrsfs}
\usepackage{url}
\usepackage{color}
\usepackage{mathtools}
\usepackage[colorlinks=true,linkcolor=blue,citecolor=blue,urlcolor=blue,breaklinks]{hyperref}
 \usepackage[foot]{amsaddr}

\allowdisplaybreaks
\mathtoolsset{showonlyrefs}

\newtheorem{theorem}{Theorem}[section]
\newtheorem{lemma}[theorem]{Lemma}

\newtheorem{remark}[theorem]{Remark}

\theoremstyle{definition}
\newtheorem{example}[theorem]{Example}

\newcommand{\be}{\begin{equation}}
\newcommand{\ee}{\end{equation}}
\newcommand{\bfig}{\begin{figure}}
\newcommand{\efig}{\end{figure}}
\newcommand{\bt}{\begin{table}}
\newcommand{\et}{\end{table}}
\newcommand{\bc}{\begin{center}}
\newcommand{\ec}{\end{center}}
\newcommand{\ba}{\begin{array}}
\newcommand{\ea}{\end{array}}
\newcommand{\bes}{\begin{equation*}}
\newcommand{\ees}{\end{equation*}}

\newcommand{\sss}{\eta}

\newcommand{\N}{\mathbb{N}}

\newcommand{\R}{\mathbb{R}}

\newcommand{\Rn}{\mathbb{R}^{n}}

\newcommand{\irn}{\int_{\Rn}}

\newcommand{\Par}{\mathcal P}

\renewcommand{\a}{\alpha}
\renewcommand{\b}{\beta}
\renewcommand{\d}{\delta}

\newcommand{\e}{\varepsilon}
\newcommand{\g}{\gamma}

\renewcommand{\l}{\lambda}

\newcommand{\n}{\nabla}
\def\div{\nabla\cdot}

\renewcommand{\t}{\tau}

\newcommand{\p}{\partial}

\DeclareMathOperator{\supp}{supp}
\DeclareMathOperator{\card}{card}

\DeclareMathOperator{\diam}{diam}

\begin{document}

\title{Geometry of minimizers for the interaction energy \\ with mildly repulsive potentials}
\author{J. A. Carrillo}
\address{Department of Mathematics, Imperial College London, South Kensington Campus, London SW7 2AZ, UK.}
\email{carrillo@imperial.ac.uk}
\author{A. Figalli}
\address{ETH Z{\"u}rich, Department of Mathematics, R{\"a}mistrasse 101, CH-8092 Z{\"u}rich, Switzerland.}
\email{alessio.figalli@math.ethz.ch}
\author{F. S. Patacchini}
\address{Department of Mathematics, Imperial College London, South Kensington Campus, London SW7 2AZ, UK.}
\email{f.patacchini13@imperial.ac.uk}

\date{\today}

\maketitle

\begin{abstract}
We show that the support of any local minimizer of the interaction energy consists of isolated points whenever the interaction potential is of class $C^2$ and mildly repulsive at the origin; moreover, if the minimizer is global, then its support is finite. In addition, for some class of potentials we prove the validity of a uniform upper bound on the cardinal of the support of a global minimizer. Finally, in the one-dimensional case, we give quantitative bounds.
\end{abstract}

\section{Introduction}
\label{sec:intro}

Consider the \emph{interaction energy} $E \colon \Par(\Rn) \to \R \cup \{+\infty\}$, defined on the set of Borel probability measures $\Par(\Rn)$ by
\begin{equation} \label{eq:energy}
	E(\mu) = \frac 12 \irn\irn W(x-y) \, d\mu(x)\,d\mu(y) \quad \mbox{for all $\mu \in \Par(\Rn)$},
\end{equation}
where $W\colon \Rn \to \R\cup\{+\infty\}$ is an \emph{interaction potential}. The study of local and global minimizers of energies of the form \eqref{eq:energy} has in recent years been in the spotlight of applied mathematics, in particular in the context of the variational approach to partial differential equations. The main reason for this interest is that $E$ is a Lyapunov functional for the continuity equation
\begin{equation*}
	\partial_t\mu = \div (\mu \n W * \mu) \quad \mbox{on $\Rn$, for $t>0$},
\end{equation*}
called the \emph{aggregation equation}, where $*$ is the convolution operator and $\mu\colon [0,\infty) \to \Par(\Rn)$ is here a probability curve. These equations describe the continuum behavior of agents interacting via the potential $W$, and are at the core of many applications ranging from mathematical biology to granular media and economics, see \cite{TBL,mogilner1999non,HP,Tos,BC} and the references therein. They can also be obtained as dissipative limits of hydrodynamic equations for collective behavior \cite{LT}.

Typically, interaction potentials are repulsive towards the origin and attractive towards infinity; this reproduces the ``social", or natural, behavior of the agents that are usually considered in applications. In \cite{BCLR} the authors showed that the dimension of the support of a minimizer of $E$ is directly related to the repulsiveness of the potential at the origin, i.e., to the strength of the repulsion of two very close particles. More precisely, the stronger the repulsion (up to Newtonian), the higher the dimension of the support. In particular, in the case of \emph{mild repulsion}---when the potential behaves like a power of order $\a$, with $\a > 2$, near the origin---the Hausdorff dimension of each smooth enough component of the support has to be zero, see \cite[Theorem 2]{BCLR}. The smoothness assumption on the connected components of the support was essential in the proof, hence several open problems immediately arise: Is it possible to have minimizers whose supports lie on sets of non-integer Hausdorff dimension? Or can the support have integer dimension but non-smooth components? And, assuming one can prove that it is of zero dimension, is the support discrete?

In this work we give a conclusive answer to all these questions. Let us first mention that extensive simulations \cite{d2006self,ring1,ring2,BCLR,ABCV,CH} showed that fractal supports were not numerically observed and, moreover, these numerical simulations were consistently giving minimizers supported on finite numbers of points. This is precisely the rigorous result we show in this work under suitable assumptions on the repulsive-attractive potential $W$. In the rest of the paper we always assume that $W$ is radially symmetric, of class $C^2$, and such that
\be\label{eq:hyp-basics}
	\mbox{$W(0)=0$, and there is $R > 0$ s.t. $W(x) < 0$ if $0< |x| < R$, and $W(x) \geq 0$ if $|x| \geq R$.}
\ee
Let us write $w(|x|):=W(x)$ for all $x\in\Rn$. Remark that, by convention, $w$ being repulsive (resp. attractive) at $r>0$ means $w'(r)<0$ (resp. $w'(r)>0$). We suppose that $W$ is mildly repulsive, that is,
\be\label{eq:hyp-power}
	\mbox{there exist $\a>2$ and $C>0$ such that $w'(r)r^{1-\a} \to -C$ as $r\to0$}.
\ee
Note that, since $w(0)=0$, \eqref{eq:hyp-power} implies that $w(r)r^{-\a} \to -C/\a$ as $r\to0$.

When we refer to minimizers of the energy \eqref{eq:energy}, we either refer to \emph{global} minimizers, in which case no underlying topology is required, or to \emph{local} minimizers, in which case we need to specify the topology. In \cite{BCLR} it was proven that a natural topology to obtain suitable Euler-Lagrange conditions is that induced by the \emph{$\infty$-Wasserstein distance}. We define the $\infty$-Wasserstein distance between two probability measures $\mu$ and $\nu$ by
\be\label{eq:wass}
	d_\infty(\mu,\nu) = \inf_{\pi \in \Pi(\mu,\nu)} \sup_{(x,y) \in \supp\pi} |x-y|,
\ee
where $\Pi(\mu,\nu)$ is the space of probability measures on $\Rn\times\Rn$ with first marginal $\mu$ and second marginal $\nu$. We therefore say that $\mu \in \Par(\Rn)$ is a \emph{$d_\infty$-local minimizer} of $E$ if there exists $\e > 0$ such that $E(\mu) \leq E(\nu)$ for all $\nu \in \Par(\Rn)$ with $d_\infty(\mu,\nu) < \e$. We refer to \cite{BCLR,CDM,Villani} and the references therein for a good account on the properties of this distance and its relation to more classical metrics in optimal transport.

The Euler-Lagrange conditions for $d_\infty$-local minimizers of $E$ were used to give necessary and sufficient conditions on repulsive-attractive potentials to have existence of global minimizers \cite{CCP,SST}, and to analyze the regularity of the $d_\infty$-local minimizers for potentials which are as repulsive as, or more singular than, the Newtonian potential \cite{CDM}. In both cases, $d_\infty$-local minimizers are solutions of some related obstacle problems for Laplacian or nonlocal Laplacian operators, implying that they are bounded and smooth in their supports, or even continuous up to the boundary \cite{CDM,CV}. Similar Euler-Lagrange equations were also used for nonlinear versions of the Keller-Segel model in order to characterize minimizers of related functionals \cite{CCV}.

Our main theorem in this work is the following.
\begin{theorem} \label{thm:atomic}
	Let $\mu\in\Par(\Rn)$ be a $d_\infty$-local minimizer of the interaction energy with $E(\mu) < \infty$. Then every point in the support of $\mu$ is isolated. If moreover $\mu$ is a global minimizer, then the support of $\mu$ consists of finitely many points.
\end{theorem}

In Section \ref{sec:preliminary} we show some preliminary results concerning minimizers that are needed for the proof of Theorem \ref{thm:atomic}, which we give in Section \ref{sec:proof-main}. In Section \ref{sec:estimate} we prove upper estimates on the cardinal of the support of a global minimizer.


\section{Preliminary results}\label{sec:preliminary}
We give in this section the preliminary results needed to prove Theorem \ref{thm:atomic}. For every $\mu \in \Par(\Rn)$, let us write
$$
V_\mu(x) := W\ast\mu(x) = \irn W(x-y)\,d\mu(y) \quad \mbox{for all $x \in \Rn$}.
$$
By the continuity of $W$ it can be checked that, for any $\mu \in \Par(\Rn)$, $V_\mu\colon \Rn \to [0,+\infty]$ is lower semi-continuous. From \cite[Proposition 2]{BCLR} we have the following lemma.
\begin{lemma} \label{lem:min}
Let $\mu$ be a $d_\infty$-local minimizer of the interaction energy with $E(\mu)<\infty$. Then there exists $\e > 0$ such that any point $x_0 \in \supp\mu$ is a local minimizer of $V_\mu$ for the radius $\e$, that is, $V_\mu(x_0)\leq V_\mu(x)$ for every $x \in B(x_0,\e)$.
\end{lemma}

\begin{remark}
	The $\e$ in the definition of a $d_\infty$-local minimizer (see below \eqref{eq:wass}) and the one in Lemma \ref{lem:min} are the same. This directly comes from the proofs of \cite[Propositions 1 and 2]{BCLR}. 
\end{remark}
By \cite[Remark 3, Theorem 4]{BCLR} and the continuity of $W$ we get the result below.
\begin{lemma}\label{lem:v-mu}
Let $\mu$ be a $d_\infty$-local minimizer with $E(\mu) < \infty$. Let $K \subset \supp\mu$ be a bounded connected set. There exists $C_K\in \R$ such that the Euler-Lagrange condition writes
$$
V_\mu(x) = C_K \quad \mbox{for all $x \in K$}.
$$
If moreover $\mu$ is a global minimizer, then we know that $C_K = 2E(\mu)$, which is independent of $K$, and we have 
\begin{equation} \label{eq:euler}
\begin{cases} V_\mu(x) = 2E(\mu) & \mbox{for all $x \in \supp\mu$},\\
V_\mu(x) \geq 2E(\mu) & \mbox{for all $x\in\Rn$}. \end{cases}
\end{equation}
\end{lemma}

We now give the computation of the second variation of the energy functional \eqref{eq:energy}, which is a crucial point of the proof of Theorem \ref{thm:atomic}.

\begin{lemma} \label{lem:min-condition}
Let $\mu$ be a $d_\infty$-local minimizer of the interaction energy with $E(\mu) < \infty$. There exists $\d > 0$ such that for all $x_0 \in \supp\mu$ we have
\begin{equation} \label{eq:min-condition}
	\irn\irn W(x-y) \, d\nu(x)\,d\nu(y) \geq 0
\end{equation}
for any signed measure $\nu$ verifying $\supp\nu \subset \supp\mu \cap B(x_0,\d)$ and $\nu(\Rn) = 0$.
\end{lemma}
\begin{proof}
Let $\e > 0$ be the constant of Lemma \ref{lem:min} for the local minimizer $\mu$. Fix $x_0 \in \supp\mu$ and $\delta \leq \e/2$, and define 
$$
\mu':=\mu_0+\mu_1,
$$
where $\mu_0$ is an arbitrary nonnegative measure with
$$
\mu_0(\Rn) =\mu(B(x_0,\d)), \qquad \supp\mu_0\subset \supp\mu\cap B(x_0,\d),
$$
and $\mu_1$ is defined as
$$
\mu_1(A):=\mu(A\setminus B(x_0,\d)) \qquad \mbox{for any Borel set $A \subset \Rn$}.
$$
Clearly $\mu' \in \Par(\Rn)$ and $d_\infty(\mu,\mu') < \e$. Let us write $B_0:=B(x_0,\d)$. Then, since $E(\mu) \leq E(\mu')$,
\be\label{eq:min}
\begin{split}
	&\int_{B_0} \int_{B_0} W(x-y)\,d\mu(y)\,d\mu(x)
+2 \int_{B_0} \int_{\Rn\setminus B_0} W(x-y)\,d\mu(y)\,d\mu(x)\\
	&\phantom{{}={}}\leq\int_{B_0} \int_{B_0} W(x-y)\,d\mu_0(y)\,d\mu_0(x)
+2 \int_{B_0} \int_{\Rn\setminus B_0} W(x-y)\,d\mu(y)\,d\mu_0(x).
\end{split}
\ee
By Lemma \ref{lem:min}, 
$$
V_\mu(x) \leq V_\mu(y) \quad \text{for all $x \in \supp\mu\cap B_0$ and $y \in B_0$}.
$$
Then, since $\mu_0(\Rn) = \mu(B_0)$, it follows from integration against $\int_{B_0}\int_{B_0}\, d\mu_0(x)\,d\mu(y)$ that
\bes
\begin{split}
	&\int_{B_0}\int_{\Rn\setminus B_0} W(x-y)\,d\mu(y)\,d\mu_0(x)-\int_{B_0}
\int_{\Rn\setminus B_0} W(x-y)\,d\mu(y)\,d\mu(x)\\
	&\phantom{{}={}}\leq 
\int_{B_0}\int_{B_0} W(x-y)\,d\mu(y)\,d\mu(x)-\int_{B_0}
\int_{ B_0} W(x-y)\,d\mu(y)\,d\mu_0(x).
\end{split}
\ees
Hence \eqref{eq:min} implies
\bes
\begin{split}
&2\int_{B_0}
\int_{ B_0} W(x-y)\,d\mu(y)\,d\mu_0(x)\\
&\phantom{{}={}}\leq \int_{B_0}\int_{B_0} W(x-y)\,d\mu(y)\,d\mu(x)
+\int_{B_0}\int_{B_0} W(x-y)\,d\mu_0(y)\,d\mu_0(x).
\end{split}
\ees
Then the above inequality becomes
$$
\int_{B_0}\int_{B_0} W(x-y)\,d[\mu-\mu_0](y)\,d[\mu-\mu_0](x) \geq 0,
$$
which, by the arbitrariness of $\mu_0$, proves \eqref{eq:min-condition} for any signed measure $\nu$ with $\supp\nu \subset \supp\mu \cap B(x_0,\d)$, $\nu(\R^n) =0$ and $|\nu|\leq \mu$. Relaxing the last condition follows by approximation. More precisely, by bilinearity \eqref{eq:min-condition} remains true also when $|\nu|\leq M\mu$ for some $M<\infty$. Then, if $\nu$ is an atomic measure of the form $\nu=\sum_{i=1}^ka_i\delta_{x_i}$, for some $k\in\N$, with $x_i \in \supp\mu \cap B(x_0,\d)$ and $a_1+\cdots+a_k=0$, it suffices to approximate it with the restricted family of measures
$$
\nu_\eta:=\sum_{i=1}^k a_i \frac{\mu \lfloor_{B(x_i,\eta)}}{\mu(B(x_i,\eta))}, \quad \eta>0,
$$
to apply \eqref{eq:min-condition} to $\nu_\eta$ (since indeed one can easily see that $|\nu_\eta|\leq M_\eta\mu$ for some $M_\eta<\infty)$, and to let $\eta\to 0$. Finally, if $\nu$ is any signed measure verifying $\supp\nu \subset \supp\mu \cap B(x_0,\d)$ and $\nu(\R^n)=0$, then one can approximate it with atomic measures supported inside $\supp\mu \cap B(x_0,\d)$.
\end{proof}

\begin{remark} \label{rem:min-condition-global}
In Lemma \ref{lem:min-condition}, if $\mu$ is a global minimizer of the interaction energy, then the result is slightly stronger. Indeed, in this case, for all $x_0 \in \supp\mu$ we have
\begin{equation} \label{eq:min-condition-global}
	\irn\irn W(x-y) \, d\nu(x)\,d\nu(y) \geq 0
\end{equation}
for any signed measure $\nu$ verifying $\supp\nu \subset \supp\mu$ and $\nu(\Rn) = 0$. This can be seen simply by adapting the proof of Lemma \ref{lem:min-condition} to global minimizers.
\end{remark}

Remark \ref{rem:min-condition-global} yields the following lemma.
\begin{lemma} \label{lem:compact-support}
	Let $\mu$ be a global minimizer of the interaction energy. Then $\mu$ is compactly supported with $\diam(\supp \mu) \leq R$.
\end{lemma}

\begin{proof}
	Let $x,y\in\supp\mu$. Take $\nu = \delta_x - \delta_y$ in \eqref{eq:min-condition-global} and get that $W(x-y) \leq 0$. Hence the result follows from the assumption \eqref{eq:hyp-basics}.
\end{proof}


\section{Proof of Theorem \ref{thm:atomic}}\label{sec:proof-main}

Suppose that $\mu$ is a $d_\infty$-local minimizer of the interaction energy, and let $y_0 \in \supp\mu$. Then, given $\d>0$ smaller than the one in Lemma \ref{lem:min-condition}, either there are two different points $y_1,y_2 \in \supp\mu \cap B(y_0,\d)$, or there is nothing to prove.

\subsubsection*{Step 1.- Geometric constraint} We show that there is a geometric constraint for the points in the support of $\mu$ stemming from the minimality condition of Lemma \ref{lem:min-condition}. For convenience, let us write $y_0=0$, $y_1=x_1$ and $y_2=-x_2$, and consider the measure
$$
\nu_\l=-\delta_{0}+\l \delta_{x_1} +(1- \l) \delta_{-x_2} \quad \mbox{for any $\l \in [0,1]$.}
$$
By plugging $\nu_\l$ in \eqref{eq:min-condition} in place of $\nu$ we get
$$
\l(1-\l) W(x_1+x_2) \geq \l W(x_1)+(1-\l)W(x_2).
$$
By the assumption in \eqref{eq:hyp-basics}, we can choose $\d$ small enough such that $W(x_1),W(x_2)$ and $W(x_1+x_2)$ are negative. Then, set $a:=W(x_1)/W(x_1+x_2)$ and $b:=W(x_2)/W(x_1+x_2)$, so that the above inequality is equivalent to
\begin{equation} \label{eq:lambda-ab}
\l(1-\l) \leq \l a+(1-\l)b\qquad \forall\,\l\in[0,1].
\end{equation}
and implies
\begin{equation} \label{eq:ab}
\sqrt{a}+\sqrt{b}\geq 1.
\end{equation}
To see this, assume $|b-a|\leq 1$ (otherwise the inequality is trivial since $a,b \geq 0$ and therefore either $a$ or $b$ would be larger than $1$). Then, choosing $\l=(1+b-a)/2$, the inequality \eqref{eq:lambda-ab} becomes, after some rearrangement,
$$
(1-a)^2-2(1+a)b+b^2 \leq 0.
$$
This implies in particular that
$$
b \geq 1+a-2\sqrt{a} = (1-\sqrt{a})^2,
$$
which is the desired inequality \eqref{eq:ab}. Now, recalling the definition of $a$ and $b$, we get, from \eqref{eq:ab},
$$
\sqrt{-W(x_1)}+\sqrt{-W(x_2)} \geq \sqrt{-W(x_1+x_2)}.
$$
Assume, by homogeneity, that $x_1+x_2=pe_1$, where $e_1$ is the first unit vector of the orthonormal basis of $\Rn$, and $p > 0$ is a small rescaling parameter. Then the above inequality becomes
$$
\sqrt{-W(x_1)}+\sqrt{-W(pe_1 - x_1)} \geq \sqrt{-W(pe_1)}.
$$
Write $x_1=p(te_1 + y)$, where $y\in\Rn$ has zero first coordinate, and, by homogeneity, $t \in [0,1]$. Then, using that $|x_1|\leq pt+p|y|$ and $|pe_1-x_1|\leq p(1-t)+p|y|$, and that, for $p$ small enough, $x\mapsto\sqrt{-W(px)}$ is radially non-decreasing in $B(0,1+|y|)$, we get
$$
\sqrt{-w(p(t + |y|))} + \sqrt{-w(p(1-t+|y|))} \geq \sqrt{-w(p)}.
$$
Write $w_p(r) := w(pr)p^{-\a}$ for any $r\geq0$. Then, dividing the inequality above by $p^{\a/2}$ yields
$$
\textstyle{\sqrt{-w_p(t + |y|)} + \sqrt{-w_p(1-t+|y|)} \geq \sqrt{-w_p(1)}.}
$$
Observe that since the inequality above is invariant under the transformation $t \leftrightarrow 1-t$, we could have assumed $t\in [0,1/2]$ without loss of generality when writing $x_1=p(te_1+y)$. Define, for all $s \in [0,1]$ and $z \in \Rn$, and for $p>0$ small,
\bes
   \sss_p(s,z) =\textstyle{ \sqrt{-w_p(s + |z|)} + \sqrt{-w_p(1-s+|z|)} - \sqrt{-w_p(1)}.}
\ees
For any two $v,v'\in\Rn$ distinct, define the open set
\bes
	\mathcal{S}_p(v,v') := \left\{u\in\Rn : \pi_{(v,v')}u \in [v,v'] \;\; \mbox{and}\;\; \sss_p\left(\frac{|\pi_{(v,v')}  u - v|}{|v-v'|},\pi_{(v,v')} u - u\right) < 0 \right\},
\ees
where $\pi_{(v,v')}$ denotes the orthogonal projection on the line $(v,v')$.

Since the point $y_0$, taken above to be the origin, can be an arbitrary point of the support of $\mu$, we just proved the validity of the following geometric constraint: for any $y_0, y_1 \in \supp\mu$ such that $|y_0-y_1| = p$ with $p$ small, we have 
\bes
	\supp\mu \cap \mathcal{S}_p(y_0,y_1) =\emptyset. 
\ees

\subsubsection*{Step 2.- Asymptotic geometric constraint} Define, for all $s \in [0,1]$ and $z \in \Rn$,
\begin{equation} \label{eq:shape}
\sss_{0,\a}(s,z) = (s+|z|)^{\a/2}+(1-s+|z|)^{\a/2} - 1.
\end{equation}
For any two $v,v'\in\Rn$ distinct, define the open set
\bes
	\mathcal{S}_{0,\a}(v,v') := \left\{u\in\Rn : \pi_{(v,v')}u \in [v,v'] \;\; \mbox{and}\;\;\sss_{0,\a} \left(\frac{|\pi_{(v,v')}  u - v|}{|v-v'|},\pi_{(v,v')} u - u\right) < 0 \right\}.
\ees
By the assumption given in \eqref{eq:hyp-power}, and assuming with no loss of generality that $C=\a$, we have $w_p(r) \to -r^\a$ and $w_p'(r) \to -\a r^{\a-1}$ as $p\to0$, and
\be\label{eq:der-eta}
	\begin{cases}
		\sss_p \to \sss_{0,\a},\\
		\n \sss_p \to \n \sss_{0,\a},
	\end{cases} \mbox{pointwise as $p\to 0$}.
\ee
This shows that, for any $y_0, y_1 \in \supp\mu$ asymptotically close, we have
\begin{equation} \label{eq:empty2}
\supp\mu \cap \mathcal{S}_{0,\a}(y_0,y_1) = \emptyset
\end{equation}
The boundary $\{\sss_{0,\a}=0\}$ is represented in Figure \ref{fig:shape} for $n = 2$ and different values of $\a$. It is important to keep in mind that this shape only depends on $\a$ and not on the choice of the minimizer $\mu$.
\begin{figure}[!ht]
\centering
	\includegraphics[scale=0.65]{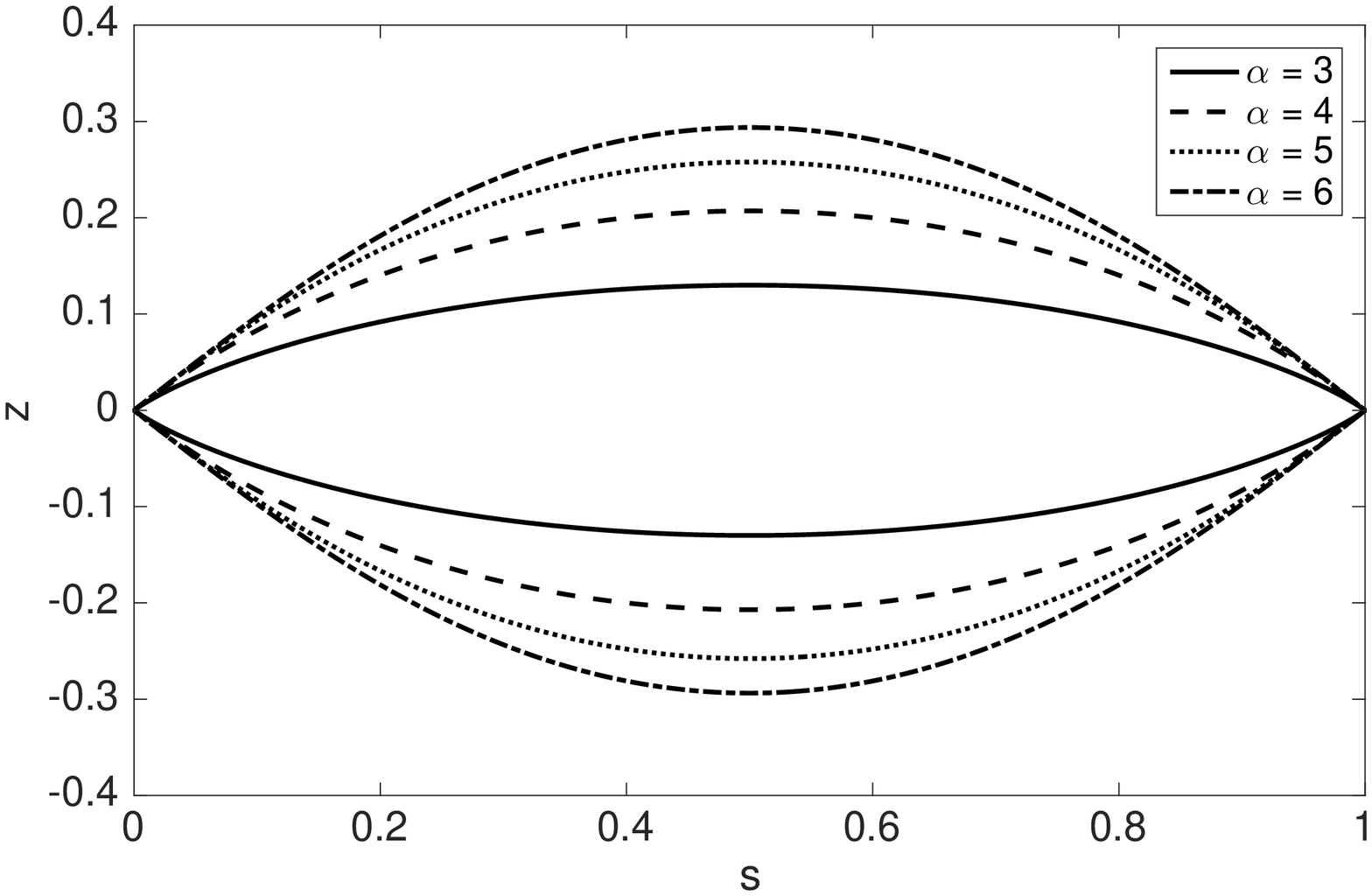}
	\caption{The boundary $\{\sss_{0,\a}=0\}$ for $n=2$ and different values of $\a$. \label{fig:shape}}
\end{figure}

\subsubsection*{Step 3.- Understanding the geometric constraint} 
Given $v,v' \in \Rn$, define the open ``double cone'' with opening $\t > 0$ by
\bes
C_\t(v,v'):=\left\{u\in\Rn : \pi_{(v,v')}u \in [v,v'] \; \mbox{and}\; |\pi_{(v,v')}u - u| < \t \min(|\pi_{(v,v')} u-v|,|\pi_{(v,v')} u-v'|) \right\}.
\ees
Since $\a>2$, then $\mathcal{S}_{0,\a}(v,v')$ is non-empty for any $v,v' \in \Rn$ distinct and $r\mapsto r^{\a/2}$ is a strictly convex function on $[0,\infty)$, which implies that $\mathcal{S}_{0,\a}(v,v')$ is a convex set. Thus, if we set
$$
\gamma_\a=\sup \{\tau > 0\,:\,C_\tau(v,v')\subset \mathcal S_{0,\a}(v,v')\},
$$
we see that $\g_\a$ can be computed using that $|z| = \g_\a s$ when $s = 1/2$ and $\sss_{0,\a}(s,z) = 0$ in \eqref{eq:shape}, and therefore we get
\bes
\g_\a = 2^{1-2/\a} - 1.
\ees
Note that $\g_\a$ only depends on $\a$ (and not on the choice of $v$ and $v'$). Therefore, for any $y_0,y_1 \in \supp \mu$ asymptotically close, \eqref{eq:empty2} gives
\bes
\supp\mu \cap C_{\g_\a}(y_0,y_1)=\emptyset.
\ees
Moreover, $\p_s \sss_{0,\a}(0,0) = -\a/2$ and $|\n \sss_{0,\a}(0,0)| = \a/\sqrt{2}$, so that no tangent plane to $\mathcal{S}_{0,\a}(v,v')$ at $s = 0$ and $z = 0$ contains $v'-v$. By symmetry, this also holds at $s=1$ and $z=0$. 

This, along with the convexity of $\mathcal{S}_{0,\a}(v,v')$ for any distinct $v,v'\in\Rn$ and \eqref{eq:der-eta}, gives us the final form of the geometric constraint needed in order to conclude: there exist $p'>0$ and $\t_{p'} < \g_\a$ such that, for all $p<p'$ and $y_0,y_1 \in \supp\mu$ with $|y_0-y_1| = p$, we have $C_{\t_{p'}}(y_0,y_1) \subset \mathcal{S}_p(y_0,y_1)$.

\subsubsection*{Step 4.- Conclusion} We now proceed by contradiction. Let us assume that $y_0$ is an accumulation point in $\supp\mu$. Then there exist a sequence $(x_k)_{k\in\N}\subset\supp\mu$ converging to $y_0$ and a unit direction $e$ such that 
$$
\frac{x_k-y_0}{|x_k-y_0|}\to e \quad \mbox{as $k\to\infty$}.
$$
This implies that if we fix $\hat k$ large enough so that $|x_{\hat k}-y_0|$ is sufficiently small and $x_{\hat k}-y_0$
is almost parallel to $e$,
then for $k \gg \hat k$ the point $x_k$ belongs to the cone 
$C_{\t_{p'}}(y_0,x_{\hat k})$,
a contradiction to Steps 2 and 3.

This proves that every point $y_0\in\supp\mu$ is isolated.
Finally, when $\mu$ is a global minimizer,
we know from Lemma \ref{lem:compact-support} that $\supp\mu$ is compact, hence $\supp\mu$ must be finite.

\begin{remark} \label{rem:gradient}
Let $n=2$. Then, $\n \sss_{0,\a}\to \a(-1/2,\pm1/2)$ at the origin, meaning that
the angle that 
tangent planes to the shape $\mathcal{S}_{0,\a}(x,y)$ make with the base of the shape at $s=z=0$ is $\pi/4$. By symmetry, this is also true at $s=1$ and $z=0$. 
\end{remark}


\section{Estimate on the cardinal of the support of a global minimizer}\label{sec:estimate}

We give here upper estimates on the cardinal of the support of a global minimizer of the interaction energy. We obtain a non-quantitative bound in any dimension, and a quantitative one in the one-dimensional case.

For any $m\in\N$ and unit direction $e$, we write
$$
D_e^m W(x) :=
\frac{d^m}{d\epsilon^m}\Big|_{\epsilon=0}W(x+\epsilon e).
$$
Note that, for every unit direction $e$, Lemma \ref{lem:min} implies that for any $d_\infty$-local minimizer $\mu$,
\begin{equation} \label{eq:minimality}
	\begin{cases} D_e^1 V_\mu(x) = D_e^1 W * \mu(x) = 0,\\ D_e^2 V_\mu(x) = D_e^2 W * \mu(x) \geq 0, \end{cases} \quad \mbox{for all $x\in \supp \mu$}.
\end{equation}

\subsection{Non-quantitative estimates in arbitrary dimension}\label{subsec:nonquant}

In addition to the hypotheses on the potential $W$ given in Section \ref{sec:intro}, we also consider the following assumption:
\be\label{eq:hyp-C4}
	\substack{\mbox{$W\in C^4_{\rm loc}(\Rn\setminus \{0\})$, $W$ satisfies \eqref{eq:hyp-power} with $\a<4$,}\\ \mbox{and for any $r>0$ there exists $M_r>0$ such that, for any unit direction $e$,}\\
	\mbox{$D^4_e W(x) \leq M_r$ for all $x \in B(0,r)\setminus\{0\}$.}}
\ee

We give below an example of potential satisfying \eqref{eq:hyp-C4}.
\begin{example}
	Assume that there exists $U\in C^4(\Rn)$ such that 
\be\label{eq:example-C4}
	W(x) = -|x|^\a/\a + U(x) \quad \mbox{for all $x\in\Rn$, with $3\leq\a<4$}.
\ee
We claim that $W$ satisfies \eqref{eq:hyp-C4}. For this, compute, for all $x\in\Rn$ and unit direction $e$, 
\bes
	D_e^4W(x) = -(\a-2) \Big(3|x|^{\a-4} + 6(\a-4)(x\cdot e)^2 |x|^{\a-6} + (\a-4)(\a-6)(x\cdot e)^4 |x|^{\a-8} \Big) + D_e^4U(x).
\ees
Therefore $W\in C^4(\Rn\setminus \{0\})$ and $D_e^4W$ blows up at the origin (since $\a<4$). We thus want the term within parentheses to be non-negative. By dilation, we can assume $|x|=1$. Then, since $(x\cdot e)^2 \in [0,1]$, it is necessary and sufficient to show that
\be\label{eq:f}
	f(\xi) := 3 - 6\b \xi + \b(\b+2) \xi^2 \geq 0 \quad \mbox{for all $\xi\in[0,1]$},
\ee
where $\b:=4-\a\in(0,1]$. Now, one can check that the minimum of $f$ is reached at $\bar\xi = 1$ and $f(\bar\xi) = (\b-3)(\b-1) \geq 0$, which proves the claim. We can also show that the condition $3\leq \a$ in \eqref{eq:example-C4} is sharp. Indeed, let $2<\a<3$; then $\beta \in (1,2)$, and the minimum of $f$ is reached at $\bar\xi = 3/(\beta +2) \in (3/4,1)$ and $f(\bar\xi) = 6(1-\beta)/(\b+2) < 0$, which violates \eqref{eq:f}.
\end{example}

\begin{lemma}\label{lem:hessian}
	Suppose that $W$ satisfies \eqref{eq:hyp-C4} and let $\mu\in\Par(\Rn)$ be a global minimizer of the interaction energy. Then $D_e^2 V_\mu(x) > 0$ for all $x \in \supp \mu$ and unit direction $e$.
\end{lemma}

\begin{proof}
	By Theorem \ref{thm:atomic} we know that the support of $\mu$ is discrete. Also, by the minimality conditions \eqref{eq:minimality} we know that the Hessian of $V_\mu$ is positive semi-definite on the support of $\mu$. Let $\{x_k\}_{k\geq1} := \supp \mu$ and denote by $m_k$ the mass of particle $x_k$ for any $k\geq 1$. Without loss of generality, assume by contradiction that $D^2_e V_\mu(x_1) = 0$. We want to build a probability measure with smaller energy than $\mu$. Define, for all $\e > 0$,
$$
\mu_\e := \sum_{i\geq 2} m_i \d_{x_i} + \frac{m_1}{2}\d_{y_\e} + \frac{m_1}{2}\d_{z_\e},
$$
where $y_\e$ and $z_\e$ are distinct points belonging to the boundary of $B(x_1,\e/2)$ such that $x_\e-x_1$ and $y_\e - x_1$ are collinear to $e$, and which we can suppose to be different from $x_k$ for all $k\geq2$. Clearly $\mu_\e$ is the probability measure resulting from splitting the mass present at $x_1$ under $\mu$ and moving it evenly to two opposite points in the direction $e$ at a distance $\e/2$ from $x_1$. Then,
\bes
\begin{split}
	E(\mu_\e) &= \frac12\sum_{i\geq2}m_i\Bigg(\sum_{\substack{j\geq2\\j\neq i}} m_j W(x_i-x_j) + m_1 W(x_i-y_\e) + m_1 W(x_i - z_\e)\Bigg) + \frac{m_1^2}{4} W(y_\e - z_\e)\\
	&= E(\mu) +\frac{m_1}{2}\sum_{i\geq2}m_i \Bigg(W(x_i-y_\e) + W(x_i - z_\e) -2W(x_i-x_1)\Bigg) + \frac{m_1^2}{4} w(\e).
\end{split}
\ees
Note that, by \eqref{eq:hyp-basics}, \eqref{eq:hyp-power} and \eqref{eq:hyp-C4}, $D_e^2 W(0)=0$. Hence,
using the upper boundedness of $D_e^4 W$ (see \eqref{eq:hyp-C4}) and the bound on the diameter of $\supp\mu$ (see Lemma \ref{lem:compact-support}), a Taylor expansion yields
\bes
\begin{split}
	E(\mu_\e) &\leq E(\mu) + \frac{m_1}{8} \sum_{i\geq2}  m_i D^2_e W(x_i - x_1) \e^2 + \frac{m_1^2}{4} w(\e) + C'\e^4\\
	&= E(\mu) + \frac{m_1}{8} D^2_e V_\mu(x_1)\e^2 - \frac{m_1}{8} D^2_eW(0)\e^2 + \frac{m_1^2}{4} w(\e) + C'\e^4 = E(\mu) + \frac{m_1^2}{4} w(\e) + C'\e^4,
\end{split}
\ees
for some constant $C'\in\R$ depending on $M_R$. Then
$$
	E(\mu_\e) \leq E(\mu) - \frac{m_1^2}{4}(C/\alpha)\e^\a + C'\e^4,
$$
where $C$ is as in \eqref{eq:hyp-power}. Since $\a < 4$ we get $E(\mu_\e) < E(\mu)$ for $\e$ small enough, which contradicts the fact that $\mu$ is a global minimizer of $E$.
\end{proof}

\begin{theorem}\label{thm:uniform-estimate}
Let $(W_k)_{k\in\N}$ be a family of potentials, compact in $C^2(\Rn)$, and satisfying the assumptions in \eqref{eq:hyp-basics} and \eqref{eq:hyp-C4}
with uniform constants.
Then there exists $N\in\N$ such that any global minimizer $\mu_k\in\Par(\Rn)$ of the interaction energy for $W_k$ satisfies $\card(\supp\mu_k) \leq N$.
\end{theorem}
\begin{proof}
	We proceed by contradiction. Since $\{W_k\}_{k\in\N}$ is compact in $C^2(\Rn)$ and satisfies the assumptions in \eqref{eq:hyp-basics} and \eqref{eq:hyp-C4}
with uniform constants, it has a limit $W$ satisfying \eqref{eq:hyp-basics} and \eqref{eq:hyp-C4}. Let $\mu_k\in \Par(\Rn)$ be global minimizers of the interaction energy associated to $W_k$, and assume by contradiction that $\card(\supp\mu_k) \to\infty$ as $k\to\infty$. By Lemma \ref{lem:compact-support}, up to a translation, $\supp\mu_k \subset B(0,R)$ for all $k\in\N$. Thus, up to a subsequence, $(\mu_k)_{k\in\N}$ converges narrowly to some $\mu\in\Par(\Rn)$ as $k\to\infty$, which, by lower semi-continuity of the energy (since $W$ is bounded from below and lower semi-continuous), is a global minimizer associated to $W$. We can pick $x_k,y_k\in\supp\mu_k$ for all $k\in\N$ such that the sequences $(x_k)_{k\in\N}$ and $(y_k)_{k\in\N}$ converge to some $x \in\supp\mu$ as $k\to\infty$. Let $e_k := (x_k-y_k)/|x_k-y_k|$ for all $k\in\N$, and, up to a subsequence, define $e:= \lim_{k\to\infty} e_k$. Thanks to the first line of \eqref{eq:minimality} and the $C^2$-continuity of $W$,
\bes
	0 = \frac{D_{e_k}^1 V_{\mu_k}(x_k) - D_{e_k}^1 V_{\mu_k}(y_k)}{|x_k-y_k|} = \int_0^1 D_{e_k}^2 V_{\mu_k}(x_k + t(y_k-x_k)) \, dt \xrightarrow[k\to\infty]{} D^2_{e} V_\mu(x).
\ees
However, Lemma \ref{lem:hessian} gives $D^2_{e}V_\mu(x)>0$, which leads to the wanted contradiction.
\end{proof}

\subsection{Quantitative estimate in one dimension}\label{subsec:n1}

We fix $n=1$. In addition to the assumptions discussed in Section \ref{sec:intro}, we consider the following:
\be\label{eq:hyp-convex}
	\mbox{There exists $r \in (0,R]$ such that $\sqrt{-w}$ is strictly convex on $(0,r)$.}
\ee

\begin{lemma} \label{lem:cond-1d}
	Suppose that $W$ satisfies \eqref{eq:hyp-convex}. Let $\mu$ be a global minimizer of the interaction energy. Then any interval of length strictly less that $r$ contains at most two points of $\supp\mu$.
\end{lemma}

\begin{proof}
Without loss of generality, we take three distinct points $0,x_1,-x_2 \in \supp\mu$ with $-x_2<0<x_1$. By applying the reasoning in the proof of Theorem \ref{thm:atomic} to a global minimizer, we get
\begin{equation} \label{eq:sqrt-cond}
\sqrt{-w(x_1)}+\sqrt{-w(x_2)} \geq \sqrt{-w(x_1+x_2)}.
\end{equation}
By contradiction, suppose that $x_1+x_2\in (0,r)$. Then, since $\sqrt{-w}$ is strictly convex on $(0,r)$,
$$
	\sqrt{-w(x_1)} < \frac{x_1}{x_1+x_2}\sqrt{-w(x_1+x_2)} + \frac{x_2}{x_1+x_2}\sqrt{-w(0)},
$$
which, since $w(0) = 0$, yields
$$
	\sqrt{-w(x_1)} < \frac{x_1}{x_1+x_2}\sqrt{-w(x_1+x_2)}.
$$
Similarly, for $-x_2$,
$$
	\sqrt{-w(x_2)} < \frac{x_2}{x_1+x_2}\sqrt{-w(x_1+x_2)}.
$$
Thus, by adding up the last two inequalities,
$$
	\sqrt{-w(x_1)} + \sqrt{-w(x_2)} < \sqrt{-w(x_1+x_2)},
$$
which contradicts \eqref{eq:sqrt-cond}. Therefore $x_1+x_2 \geq r$, which is the desired result.
\end{proof}

As a direct consequence of Lemmas \ref{lem:compact-support} and \ref{lem:cond-1d}, we get the theorem below.
\begin{theorem} \label{thm:cond-1d}
	Suppose that $W$ satisfies \eqref{eq:hyp-convex}. Let $\mu\in\Par(\R)$ be a global minimizer of the interaction energy. Then $\card(\supp\mu) \leq 2\lceil R/r \rceil+1$.
\end{theorem}

We give now an example of potential for which we can apply the result above to estimate the number of points in the support of a global minimizer. 
\begin{example}
Consider the power-law potential
\be\label{eq:pl}
	W(x) = \frac{|x|^a}{a} - \frac{|x|^b}{b} \quad \mbox{for all $x\in\R$ with $2< b<a$}.
\ee
We want to use Theorem \ref{thm:cond-1d} to estimate the number of points in the support of a global minimizer $\mu$ for the potential \eqref{eq:pl} as a function of $a$ and $b$. Clearly, $W$ satisfies the assumptions in \eqref{eq:hyp-basics}, with $R=R(a,b) = (a/b)^{1/(a-b)}$, and in \eqref{eq:hyp-power}. After some computation, we get that the condition in \eqref{eq:hyp-convex} is fulfilled for $r = r(a,b)$, where
\bes
	r(a,b) = \left(\frac{a(a-1) +b(b-1) - ab}{b(a-2)}-\sqrt{\left(\frac{a(a-1) +b(b-1) - ab}{b(a-2)}\right)^2- \frac{a(b-2)}{b(a-2)}}\right)^{1/(a-b)}
\ees
Then, Theorem \ref{thm:cond-1d} gives us an upper estimate on $\card(\supp\mu)$. To make it simpler, let us assume that $a=2b$. By injecting this into the above equation, we obtain
$$
	r(a,b) = r(b) = \left(\frac32 - \frac12\sqrt{\frac{5b-1}{b-1}}\right)^{1/b}.
$$
Since now $R(a,b) = R(b) = 2^{1/b}$, Theorem \ref{thm:cond-1d} gives
$$
	\card(\supp\mu) \leq 2\left\lceil\left( \frac{4\sqrt{b-1}}{3\sqrt{b-1}-\sqrt{5b-1}} \right)^{1/b}\right\rceil+1.
$$
Note that the right-hand side of the formula is infinity for the limit case $b=2$. When $b\to\infty$, however, the right-hand side of the above formula gives an upper bound of 5 points in the support.
\end{example}

\begin{remark}
When $W$ is a power-law potential as in \eqref{eq:pl} with even powers, we can estimate the number of points in the support of a global minimizer as follows: if $a>b>0$ are even integers and $\mu$ is a global minimizer, then $V_\mu$ is a polynomial of degree $a$.
Since \eqref{eq:euler} implies that every point in $\supp\mu$ is a double root of such polynomial, this cannot happen more than $a/2$ times, so that $\card(\supp\mu) \leq a/2$. Note that this estimate, unlike the one of Theorem \ref{thm:cond-1d}, is independent of $b$ and is thus valid also for the limit case $b=2$.
\end{remark}

\subsection*{Acknowledgments}
J.A.C. was partially supported by the Royal Society via a Wolfson Research Merit Award. 
A.F. was partially supported by NSF Grants DMS-1262411
and DMS-1361122.

\bibliography{geometry_global_minimisers}

\begin{thebibliography}{10}

\bibitem{ABCV}
G.~Albi, D.~Balagu{\'e}, J.~A. Carrillo, and J.~Von~Brecht.
\newblock Stability analysis of flock and mill rings for second order models in
  swarming.
\newblock {\em SIAM J. Appl. Math.}, 74(3):794--818, 2014.

\bibitem{BCLR}
D.~Balagu{\'e}, J.~A. Carrillo, T.~Laurent, and G.~Raoul.
\newblock Dimensionality of local minimizers of the interaction energy.
\newblock {\em Arch. Ration. Mech. Anal.}, 209(3):1055--1088, 2013.

\bibitem{ring2}
A.~L. Bertozzi, T.~Kolokolnikov, H.~Sun, D.~Uminsky, and J.~Von~Brecht.
\newblock Ring patterns and their bifurcations in a nonlocal model of
  biological swarms.
\newblock {\em Commun. Math. Sci.}, 13(4), 2015.

\bibitem{BC}
A.~Blanchet and G.~Carlier.
\newblock From {N}ash to {C}ournot--{N}ash equilibria via the
  {M}onge--{K}antorovich problem.
\newblock {\em Philos. Trans. R. Soc. Lond. Ser. A Math. Phys. Eng. Sci.},
  372(2028):20130398, 11, 2014.

\bibitem{CCP}
J.~A. Ca{\~n}izo, J.~A. Carrillo, and F.~S. Patacchini.
\newblock Existence of compactly supported global minimisers for the
  interaction energy.
\newblock {\em Arch. Ration. Mech. Anal.}, 217(3):1197--1217, 2015.

\bibitem{CCV}
J.~A. Carrillo, D.~Castorina, and B.~Volzone.
\newblock Ground states for diffusion dominated free energies with logarithmic
  interaction.
\newblock {\em SIAM J. Math. Anal.}, 47(1):1--25, 2015.

\bibitem{CDM}
J.~A. Carrillo, M.~G. Delgadino, and A.~Mellet.
\newblock Regularity of local minimizers of the interaction energy via obstacle
  problems.
\newblock {\em Comm. Math. Phys.}, 343(3):747--781, 2016.

\bibitem{CH}
J.~A. Carrillo and Y.~Huang.
\newblock Explicit equilibrium solutions for the aggregation equation with
  power-law potentials.
\newblock {\em \emph{To appear in} Kin. Rel. Mod.}, 2016.

\bibitem{CV}
J.~A. Carrillo and J.~L. V{\'a}zquez.
\newblock Some free boundary problems involving non-local diffusion and
  aggregation.
\newblock {\em Philos. Trans. A}, 373(2050):20140275, 16, 2015.

\bibitem{d2006self}
M.~R. D'Orsogna, Y.-L. Chuang, A.~L. Bertozzi, and L.~S. Chayes.
\newblock Self-propelled particles with soft-core interactions: patterns,
  stability, and collapse.
\newblock {\em Phys. Rev. Lett.}, 96(10):104302, 2006.

\bibitem{HP}
D.~D. Holm and V.~Putkaradze.
\newblock Formation of clumps and patches in self-aggregation of finite-size
  particles.
\newblock {\em Phys. D}, 220(2):183--196, 2006.

\bibitem{ring1}
T.~Kolokolnikov, H.~Sun, D.~Uminsky, and A.~L. Bertozzi.
\newblock Stability of ring patterns arising from two-dimensional particle
  interactions.
\newblock {\em Phys. Rev. E}, 84(1):015203, 4, 2011.

\bibitem{LT}
C.~Lattanzio and A.~E. Tzavaras.
\newblock Relative entropy in diffusive relaxation.
\newblock {\em SIAM J. Math. Anal.}, 45(3):1563--1584, 2013.

\bibitem{mogilner1999non}
A.~Mogilner and L.~Edelstein-Keshet.
\newblock A non-local model for a swarm.
\newblock {\em J. Math. Biol.}, 38(6):534--570, 1999.

\bibitem{SST}
R.~Simione, D.~Slep{\v{c}}ev, and I.~Topaloglu.
\newblock Existence of ground states of nonlocal-interaction energies.
\newblock {\em J. Stat. Phys.}, 159(4):972--986, 2015.

\bibitem{TBL}
C.~M. Topaz, A.~L. Bertozzi, and M.~A. Lewis.
\newblock A nonlocal continuum model for biological aggregation.
\newblock {\em Bull. Math. Biol.}, 68(7):1601--1623, 2006.

\bibitem{Tos}
G.~Toscani.
\newblock One-dimensional kinetic models of granular flows.
\newblock {\em Math. Model. Numer. Anal.}, 34(6):1277--1291, 2000.

\bibitem{Villani}
C.~Villani.
\newblock {\em Topics in Optimal Transportation}.
\newblock Graduate studies in mathematics. American Mathematical Society,
  Providence (R.I.), 2003.

\end{thebibliography}
\bibliographystyle{abbrv}

\end{document}